\documentclass[letterpaper, 10 pt, conference]{ieeeconf}  
\IEEEoverridecommandlockouts
\usepackage{balance}
\usepackage{color}
\usepackage{caption}
\usepackage{enumerate}
\usepackage{cite}
\usepackage{empheq}
\usepackage{mathrsfs}

\usepackage{mathtools}

\usepackage{tikz}
\usepackage{circuitikz}

\usepackage[font=small,labelfont=bf]{caption}

\newcommand{\tUB}{\textstyle{\mathrm{UB}}}
\newcommand{\tdiam}{\text{diam}}

\newcommand{\real}{\ensuremath{\mathbb{R}}}

\newcommand{\until}[1]{\{1,\dots, #1\}}

\newcommand{\supscr}[2]{#1^{\textup{#2}}}

\usepackage{subfig}


\makeatletter
\pgfcircdeclarebipole{}{\ctikzvalof{bipoles/interr/height 2}}{spst}{\ctikzvalof{bipoles/interr/height}}{\ctikzvalof{bipoles/interr/width}}{

    \pgfsetlinewidth{\pgfkeysvalueof{/tikz/circuitikz/bipoles/thickness}\pgfstartlinewidth}

    \pgfpathmoveto{\pgfpoint{\pgf@circ@res@left}{0pt}}
    \pgfpathlineto{\pgfpoint{.6\pgf@circ@res@right}{\pgf@circ@res@up}}
    \pgfusepath{draw}   
}

\def\pgf@circ@spst@path#1{\pgf@circ@bipole@path{spst}{#1}}
\tikzset{switch/.style = {\circuitikzbasekey, /tikz/to path=\pgf@circ@spst@path, l=#1}}
\tikzset{spst/.style = {switch = #1}}
\makeatother

\makeatletter
\let\proof\@undefined                        
\let\endproof\@undefined                  
\makeatother
\usepackage{graphicx,amssymb,amstext,amsmath,amsthm}

\usepackage[bookmarks=true]{hyperref}
\usepackage{algorithm,algorithmicx,algpseudocode}
\algnewcommand{\algorithmicgoto}{\textbf{go to}}%
\algnewcommand{\Goto}[1]{\algorithmicgoto~\ref{#1}}%
\algnewcommand{\LineComment}[1]{\Statex \(\triangleright\) #1}
\algnewcommand{\LineCommentN}[1]{\Statex \hspace{1cm}\(\triangleright\) #1}

\usepackage{multirow}
\usepackage{stfloats}

\newcommand{\argmin}{\operatornamewithlimits{arg\ min}}
\newcommand{\st}{\operatorname{s.t.}}

\newtheorem{prop}{Proposition} 

\newtheorem{thm}{Theorem}
	\newtheorem{assumption}{Assumption}
\newtheorem{lem}{Lemma}
\newtheorem{defn}{Definition}

\newtheorem{problem}{Problem}

\usepackage{setspace}

\let\oldbibliography\thebibliography
\renewcommand{\thebibliography}[1]{%
  \oldbibliography{#1}%
}

\newcommand{\moh}[1]{{\color{black} #1}}

\newcommand{\mk}[1]{{\color{black} #1}}

\newcommand{\md}[1]{{\color{black} #1}}

\newcommand{\mh}[1]{{\color{black} #1}}

\newcommand{\mok}[1]{{\color{black} #1}}
\newcommand{\kj}[1]{{\color{black} #1}}
\begin{document}

\title{\LARGE \bf \md{Guaranteed Privacy of Distributed Nonconvex
    Optimization via \kj{Mixed-Monotone} Functional Perturbations}} 

\author{%
Mohammad Khajenejad and Sonia Mart{\'\i}nez\\
\thanks{
{M. Khajenejad and S. Mart{\'\i}nez are with the Department of
  Mechanical and Aerospace Engineering, University of California, San
  Diego, San Diego, CA, USA (e-mail: \{mkhajenejad,
  soniamd\}@ucsd.edu).}} 
}

\maketitle
\thispagestyle{empty}
\pagestyle{empty}

\begin{abstract}
  In this paper, we introduce a new notion of guaranteed privacy
  \kj{
    that requires that the change of the range of the corresponding inclusion function to the true
    function is small.} 
  \kj{In particular, leveraging mixed-monotone inclusion functions, we}
  propose a \mok{privacy-preserving mechanism for nonconvex
    distributed optimization, which is based on} \kj{deterministic,
    but unknown,} affine perturbation \mok{of the local objective
    functions, \kj{which is stronger than probabilistic differential privacy}}. \kj{The design requires a robust optimization method
    to characterize the best accuracy that can be achieved by an
    optimal perturbation. Subsequently, this is used to guide the
    refinement of a guaranteed-private perturbation mechanism that
    can achieve a quantifiable accuracy via a theoretical upper bound
    that is shown to be independent of the chosen optimization
    algorithm.}
 \end{abstract}
\vspace{-0.1cm}
\section{Introduction}\label{sec:intro} 
Data privacy and protection have become a critical concern in the
management of cyber-physical systems (CPS) and their public
trustworthiness.
In such applications, malicious agents can expand their attack surface
by extracting valuable information from the many physical, control,
and communication components of the system, inflicting damage on the
CPS and its users. Hence, a great effort is being devoted to design
robust data-security control strategies for these
systems {\cite{cortes2016differential}. \kj{Motivated by this, we aim to investigate an alternative design
  of privacy-preserving mechanisms, which can make the quantification
  of privacy, as well as the associated performance loss, both
  tractable and reasonable.}

{\emph{{Literature Review}.}} Among the many approaches to data
security, one can distinguish privacy-aware methods that protect
sensitive data from worst-case data breaches by adding random
perturbations to it. However, this high data-resiliency can come at
the cost of high performance loss, which either be hard to quantify in
practice, or theoretically bounded by indices that are too large to be
useful. A main \kj{approach to characterize this trade-off while
  preserving privacy} in the literature is that of \emph{differential
  privacy}~\cite{dwork2006calibrating}.
   
  This
method was originally proposed for the protection of databases of
individual records subject to public queries. \kj{In particular, a
system processing sensitive inputs is made differentially private by randomizing its answers in
such a way that the distribution over published outputs is not
too sensitive to the data provided by any single participant. This notion} 
has
been extended to several areas \kj{in machine learning and regression~ \cite{chaudhuri2011differentially,zhang2012functional,hall2013differential}, control (estimation, verification)~\cite{wang2017differential,han2021numerical}, multi-agent systems (consensus, message passing)~\cite{huang2015differentially,han2016differentially,hale2015differentially}, and optimization and games~\cite{ding2021differentially,li2020privacy,ye2021differentially}, thereafter.} 
  \kj{In particular, a notable privacy-preserving mechanism design approach in the literature is based on the idea of message
perturbation, i.e., modifying an original non-private algorithm by
having agents perturb the messages/outputs to their neighbors or a central
coordinator with Laplace or Gaussian noise~\cite{huang2015differentially,han2016differentially,hale2015differentially}. This approach benefits from working
with the original objective functions, however suffers
from a steady-state accuracy error, since for fixed design parameters, the algorithm's output does not correspond the true optimizer in the absence of noise~\cite{nozari2016differentially}.

Another notable approach to privacy relies on functional perturbation, with the idea of having agent(s) independently
perturb their objective function(s) in a differentially private way and then participating in a centralized (distributed) algorithm~\cite{zhang2012functional,chaudhuri2011differentially,nozari2016differentially}}.
 The work in
\cite{chaudhuri2011differentially} proposed
a differentially private classifier by perturbing the objective
function with a linear finite-dimensional function. 
However, 
only the privacy of the underlying finite-dimensional parameter
set--and not the entire objective functions--is preserved. \kj{A sensitivity-based differentially private algorithm was designed in~\cite{zhang2012functional}}
by \kj{perturbing the Taylor
expansion of the cost function,} 
where, unfortunately, the
functional space had to be restricted to the space of quadratic
functions. {\color{black} Similar perturbations proposed in~\cite{hall2013differential} 
but without ensuring the
smoothness and convexity of the perturbed function.} 
In addition, none of~\kj{\cite{chaudhuri2011differentially,zhang2012functional,hall2013differential} provided} a
systematic way to study the effect of added noise on the global
optimizer. 
To address this issue, \kj{~\cite{nozari2016differentially} suggested specific {functional
  perturbations} such that} 
the difference between the
probabilities of events corresponding to any pair of data sets is
bounded by a function of the distance between the data sets, \kj{at the expense of}
a trade-off between the
privacy and the accuracy of the mechanism. 
\kj{Further, the works in
  \cite{han2021numerical,wang2022verification} provided theoretically
  proven numerical methods to quantify differential privacy in high
  probability in estimation and verification, respectively.}
\kj{However}, differential privacy \kj{inherently}
requires the slight change of the \emph{statistics} of the output of the
perturbed function or message if the objective function or the message
sent from one agent changes. This, 
satisfies privacy only in a probabilistic, and not
a 
guaranteed sense. To bridge this gap, we aim
to 
\kj{investigate an alternative deterministic approach to privacy-preserving mechanism design}, \kj{and} also to
relax the convexity assumption \kj{needed in most of the existing works.} 

{\emph{{Contributions}.}} We start by introducing a novel notion of
guaranteed privacy. This notion applies to deterministic, but unknown,
functional perturbations of an optimization problem, and characterizes
privacy in terms of how close the ranges of two sets of functions in a
vicinity are.  This notion of privacy is stronger than that of
differential privacy in the sense that it guarantees that the change
of the inclusion of the true function ranges is small, as opposed to
looking at perturbed function statistics. By exploiting the
differentiability and local Lipschitzness of the objective functions,
we propose a novel perturbation mechanism that relies on the
mixed-monotone inclusion functions of the problem objectives. We then
characterize the \kj{\emph{best accuracy}} that can be achieved by an optimal
perturbation, and use this to guide the refinement of a
guaranteed-private, perturbation mechanism that can achieve a
quantifiable accuracy via a theoretical upper bound. The design
requires a robust optimization approach, and restricts privacy \kj{to}
functions in a given vicinity. Simulations are used to illustrate the
\kj{level of the tightness} of this bound for different nonconvex optimization
algorithms, illustrating the accuracy-privacy compromise.

\section{Preliminaries}
In this section, we introduce basic notation, as well as preliminary
concepts and results used in the sequel.

{\emph{{Notation}.}}  $\mathbb{R}^n,\mathbb{R}^{n \times
  p},\mathbb{D}_n,\mh{\mathbb{R}^n_{\geq 0}}$, \mh{and
  $\mathbb{R}^n_{>0}$} denote the $n$-dimensional Euclidean space, and
the sets of $n$ by $p$ matrices, diagonal $n$ by $n$ matrices, and
nonnegative and positive vectors in $\mathbb{R}^n$, respectively.
Also, $\mathbf{0}_{n \times p},I_n$ and $\mathbf{0}_{n}$ denote the
zero matrix in $\mathbb{R}^{n \times p}$, the identity matrix in
$\mathbb{R}^{n \times n}$, and the zero vector in $\mathbb{R}^n$,
respectively. Further, for $D \subseteq \mathbb{R}^{n}$, $L_2(D)$
denotes the set of square-integrable measurable functions over $D$.
Given $M \in \mathbb{R}^{n \times p}$, $M^\top$ represents its
transpose, $M_{ij}$ denotes $M$'s entry in the $\supscr{i}{th}$ row
and the $j^\textup{th}$ column, $M^{\oplus}\triangleq
\max(M,\mathbf{0}_{n \times p})$, $M^{\ominus}=M^{\oplus}-M$ and
$|M|\triangleq M^{\oplus}+M^{\ominus}$.  Finally, for $a,b \in
\mathbb{R}^n, a\leq b$ means $a_i \leq b_i, \forall i \in
\until{n}$.
\begin{defn}[Hyper-intervals]
  \mh{A (hyper-)interval {$\mathcal{I} \triangleq
      [\underline{\md{z}},\overline{\md{z}}] \subset 
      \mathbb{R}^n$}, or an $n$-dimensional interval, is the set of
    all real vectors $\md{z \in \mathbb{R}^{n}}$ that satisfy
    $\underline{\md{z}} \le \md{z} \le \overline{\md{z}}$. \kj{Moreover,} we
    call 
    $\textstyle{\mathrm{diam}}(\mathcal{I}) \triangleq
    \|\overline{\md{z}}-\underline{\md{z}}\|\mk{_{\infty}\triangleq
      \max_{i \in
        \{1,\cdots,\md{n}\}}\md{|\overline{z}_i-\underline{z}_i|}}$
    the \emph{diameter} or \emph{interval width} of
    $\mathcal{I}$. Finally, $\mathbb{IR}^n$ denotes the space of all
    $n$-dimensional intervals, i.e., \emph{interval
      vectors}. }
\end{defn}
\begin{defn}[$\mathcal{V}$-Adjacent Sets of Functions]
  Given any normed vector space $(\mathcal{V};
  \|\cdot\|_{\mathcal{V}})$ with $\mathcal{V} \subseteq L_2(D)$, two
  sets of functions $F\triangleq \{f_1,\dots,f_n\},F'\triangleq
  \{f'_1,\dots,f'_n\} \subset L_2(D)$ are called
  \kj{\emph{$\mathcal{V}$-adjacent}} if there exists $i_0\hspace{-.1cm} \in
  \until{n}$ such that
\begin{align*}
f_{i_0}-f'_{i_0} \in \mathcal{V}, \quad \text{and} \quad f_i=f'_i, \text{
  for all other } i\neq i_0\,.
\end{align*}
\end{defn}
\section{Problem Formulation} \label{sec:Problem} Consider a group of
$N$ agents \kj{communicating over a network}.
Each agent $i \hspace{-.1cm}\in\hspace{-.1cm} \until{\kj{N}}$ has a local
objective function $\hat f_i \hspace{-.1cm}: \hspace{-.1cm}D \to
\mathbb{R}$, where $D \subset \mathbb{R}^n$ has a nonempty interior.
Consider the problem
\begin{align}\label{eq:opt_init}
  \min_{x \in \mathcal{X}_0} \hat f(x) \hspace{-.1cm}\triangleq
  \hspace{-.1cm}\sum_{i=1}^N \hat f_i(x) \ \st \
  G(x)\hspace{-.1cm}\leq
  \hspace{-.1cm}0,H(x)\hspace{-.1cm}=\hspace{-.1cm}b,x \in
    \mathcal{X}_0,
\end{align}
where the mappings $G:D \to \mathbb{R}^m,H:D \to \mathbb{R}^s$
{are convex}, the vector $b \in \mathbb{R}^s$ are known to all
agents {and $\mathcal{X}_0 \triangleq
  [\underline{x}_0,\overline{x}_0]$ is assumed to be an interval in
  $\mathbb{R}^n$. By applying a  \emph{penalty} method
  \cite[Ch.5]{bertsekas2003convex},} the problem in
\eqref{eq:opt_init} is equivalent to
\begin{align}\label{eq:opt}
\min_{x \in \mathcal{X}_0} f(x) \triangleq \sum_{i=1}^N f_i(x), 
\end{align} {where}, we assume {the following}:
\begin{assumption}
[\kj{Locally Lipschitz, Differentiable and {P}rivate {O}bjective {F}unctions}]\label{ass:mixed_monotonicity} 
Each $f_i$ is \mk{differentiable}, locally Lipschitz in its domain and
only known to agent $i$. Moreover, upper and lower \mh{uniform} bounds
for its Jacobian matrix, $\overline{J}_i^{f},\underline{J}_i^{f} \in
\mathbb{R}^{1 \times \md{n}}$ are only known to agent
$i$. 
\end{assumption}
Note that we do not restrict any $f_i, i \in \until{n}$, to be convex,
nor twice-continuously differentiable. 
  It is
assumed that the problem constraint set \kj{$\mathcal{X}_0$}
  is globally known to each agent.
The problem objective is to define a mechanism so that agents can
solve \eqref{eq:opt} via a distributed, private algorithm. 

It is known that even if $f_i$ and/or its gradient may not be
directly shared among agents, an adversary may be able to infer it by
compounding side information with network communications. This problem
has been addressed via the notion of differential privacy and the
design of differentially-private algorithms, e.g.,
in~\cite{nozari2016differentially,ding2021differentially,
  li2020privacy,ye2021differentially}, \kj{for convex objective functions} (cf.~Section~\ref{sec:intro}
for a thorough review of the relevant work). 
Differential privacy applies a random, additive perturbation to either
the algorithm input data or its output, so that the result becomes
very close to that of the same process when applied to data in a
vicinity. Here, our problem data refers to the problem objective
functions, and the \kj{novel to-be-designed} privacy mechanism consists of applying
deterministic additive perturbations characterized by intervals---but
unknown to \kj{the} adversary. By doing so, the \kj{approximated} range of the
objective functions is also an interval, and privacy can be measured
by how close these intervals are for functions in a  vicinity.
 We formalize this via a mapping $\mathcal{M}$, as
 follows, and provide one such mapping in the next section. 
\begin{defn}[Guaranteed Privacy]\label{defn:guar_privacy}
  Let $\mathcal{M}:L_2(D)^{N} \times \mathbb{IR}^n \to \mathbb{IR}^N$
  be a deterministic interval-valued map from the function space
  $L_2(D)^{N}$ to the space of intervals in $\mathbb{R}^N$.  Given
  {\color{black} $\mathcal{X} \triangleq
    [\underline{x},\overline{x}] \in \mathbb{IR}^n$}
    and a
  ``privacy gap'' $\epsilon \in \mathbb{R}_{>0}$, the map
  $\mathcal{M}$ is \kj{\emph{$\epsilon$-guaranteed private}} 
    with respect to the vicinity $\mathcal{V}$, if for any
  two $\mathcal{V}$-adjacent sets of functions \kj{$F\triangleq \{f_i\}_{i=1}^N$} and \kj{$F'\triangleq \{f'_i\}_{i=1}^N$} that (at
  most) differ in their
  $i_0^\textup{th}$ 
  element and any interval \kj{$\mathcal{I} 
  \in  \mathbb{IR}^N$} 
    such that $\mathcal{M}(F,\kj{\mathcal{X}})
  \subseteq \mathcal{I}$, one has
\begin{align}\label{eq:guaranteed_privacy_inequality}
  \hspace{-.2cm}\textstyle{\mathrm{diam}}{(\mathcal{M}(F',\kj{\mathcal{X}})\hspace{-.05cm}\cap\hspace{-.05cm}
    \mathcal{I})} \hspace{-.1cm}\leq
  \hspace{-.1cm}e^{\epsilon\|f_{i_0}\hspace{-.1cm}-\hspace{-.1cm}f'_{i_0}\|_{\mathcal{V}}}\textstyle{\mathrm{diam}}{(\mathcal{M}(F,\kj{\mathcal{X}}))}.
\end{align}
\end{defn}
It is worth to \kj{re}emphasize the difference between the
notions of guaranteed privacy (introduced above) and differential
privacy
utilized in most of the work in the literature, e.g.,
\cite{nozari2016differentially,ding2021differentially,li2020privacy}.
When differential privacy is considered, the \emph{statistics} of the
output of $\mathcal{M}$, i.e., the \emph{probability} of the value of
$\mathcal{M}$ belonging to some set changes only relatively slightly
if the objective function of one agent changes, and the change is in
$\mathcal{V}$ (cf. \cite[Definition III.1]{nozari2016differentially}
for more details). On the other hand, when guaranteed privacy is
considered, instead of introducing randomness to disguise the
objective function, we implement a range perturbation \kj{of}
$\mathcal{V}$-adjacent functions as defined above, with the end goal
of robustifying the optimization problem \kj{\eqref{eq:opt}} in a controlled manner by an
$\epsilon$ gap. 
This being said, note that our goal is to design a new mechanism (or
mapping) $\mathcal{M}$ that preserves the $\epsilon$-guaranteed
privacy of the objective functions with respect to the problem
solution, regardless of the distributed (nonconvex) optimization
algorithm chosen to arrive at it. In this case, $\mathcal{M}$ can be
interpreted as a {\color{black} preventive }action on the set of
local functions $F$, {\color{black} which will guarantee the privacy
  of \emph{any} distributed non-convex optimization algorithm applied to \kj{solve \eqref{eq:opt}.}}
  So, our problem  
can be cast as \kj{follows}:
\begin{problem}[Guaranteed Privacy-Preserving Distributed Optimization]
  Given the program in \eqref{eq:opt}, \kj{design a mechanism (map) $\mathcal{M}$ that maintains the privacy of {any} convergent distributed nonconvex optimization algorithm in the sense of Definition  \ref{defn:guar_privacy}. In other words, $\mathcal{M}$ is an $\epsilon$-guaranteed privacy-preserving map with some desired $\epsilon >0$. Moreover, the mechanism's guarantee on accuracy
  improves as the level of privacy decreases}.
\end{problem}
\vspace{-0.25cm}
\section{Guaranteed Privacy-preserving Functional
  Perturbation} \label{sec:observer} In this section, we introduce our
proposed strategy to design a guaranteed privacy-preserving mechanism
(or mapping) for distributed nonconvex optimization. The main idea is
to perturb the true objective function by \kj{deterministic but unknown}
\emph{linear additive perturbation function{s}} {in a distributed manner,} such that \kj{the} privacy is
preserved, regardless of the utilized optimization algorithm. Then, we
show that the level of privacy can be estimated by computing the
over-approximation of the range of the true and perturbed functions
using mixed-monotone inclusion functions. For the sake of
completeness, we start by briefly recapping the notions of inclusion
and decomposition functions and mixed-monotonicity that will be used
throughout our main results.
\begin{defn}[Inclusion Functions]\cite[Chapter
  2.4]{jaulinapplied}\label{defn:In_Fun}
  Consider a function ${g}: \mathbb{R}^{n} \to \mathbb{R}^{N}$. The
  interval function $[{g}]:\mathbb{IR}^{n} \hspace{-.1cm}\to
  \hspace{-.1cm} \mathbb{IR}^{N}$ is an inclusion function for ${g}$,
    if
  \begin{align*}
    \forall \mathcal{X} \in \mathbb{IR}^{n}, {g}(\mathcal{X}) \subseteq
    [{g}](\mathcal{X}),
  \end{align*}
  where ${g}(\mathcal{X})$ denotes the true image set {(or range)} of
  ${g}$ 
  {for the domain} $\mathcal{X} \in \mathbb{IR}^{n}$. 
\end{defn}
\begin{prop}[JSS Decomposition]\cite[Corollary
  2]{khajenejad2021tight}\label{prop:JSS_decomp} 
    Let $\mok{g}
  :\mathcal{X} \triangleq [\underline{x},\overline{x}] \in
  \mathbb{IR}^{n} \to \mathbb{R}$ and suppose $\forall x \in
  \mathcal{X}, J^g(x) \in [\underline{J}^g,\overline{J}^g]$, where
  $J^g(x)$ is the gradient of $\mok{g}$ at $x$ and
  $\underline{J}^g,\overline{J}^g$ are known row \mok{vectors} in
  $\mathbb{R}^{1 \times n}$. Then, $g$ can be decomposed into the sum
  of an affine mapping, {\color{black} parameterized by
    a row vector $m \in \real^{1 \times m} \in \mathbf{M}_g$ },
  and a \mok{remainder} mapping {\color{black} $h:
    \mathcal{X} \rightarrow \real$:} 
\begin{align}\label{eq:JSS_decomp}
 \mok{g}(x)=h(x)+mx,  \quad \forall x \in \mathcal{X}, \ \text{where}
\end{align}
\begin{align}\label{eq:H_decomp}
  \mathbf{M}_g \hspace{-.1cm}\triangleq \hspace{-.1cm} \{m\mok{\in \mathbb{R}^{1 \times m}} | m_{j}\hspace{-.1cm}=\hspace{-.1cm}\overline{J}^g_{j} \ \text{ or } \
  m_{j}\hspace{-.1cm}=\hspace{-.1cm}\underline{J}^g_{j}, \;1 \leq j \leq n \}.
\end{align}  
Further, $h$ is a \mok{\emph{Jacobian sign-stable}} (JSS)
{\cite{yang2019sufficient}}
mapping in $\mathcal{X}$ by construction, i.e., its Jacobian \mok{vector}
  entries have constant sign
  over $\mathcal{X}$. Therefore, for each  $j \in \until{n}$,
  either of the
following hold:
\begin{align*}
   J^{h}_{j}(\md{x}) \geq 0, \quad \forall \md{x \in \mathcal{X}},\;\; \text{or} \
  J^{h}_{j}(\md{x}) \leq 0, \quad \forall \md{x \in \mathcal{X}},
\end{align*} 
where $J^{h}(\md{x})$ denotes the Jacobian \mok{vector}
of $h$ at $\md{x \in
  \mathcal{X}}$. 
\end{prop}
\begin{prop}[Mixed-Monotone Inclusion
  Functions]\cite[Proposition 4]{khajenejad2021guaranteed}\label{prop:mm_inclusion} 
    Given the assumptions in Proposition
  \ref{prop:JSS_decomp},
\begin{align}
  \hspace{-.2cm}[\mok{g}]
  (\mathcal{X})\hspace{-.1cm}=\hspace{-.1cm}[h_d(\underline{x},\overline{x})
  \hspace{-.1cm}+\hspace{-.1cm}m^\oplus\underline{x}\hspace{-.1cm}-\hspace{-.1cm}m^\ominus
  \overline{x}, h_d(\overline{x},\underline{x})\hspace{-.1cm}+
  \hspace{-.1cm}m^\oplus\overline{x}\hspace{-.1cm}-\hspace{-.1cm}m^\ominus
  \underline{x}],
\end{align} 
with $h_{d}({x}_1,{x}_2) \triangleq
h(\mok{B}{x}_1+(I_{n}-\mok{B}){x}_2)$, {for any {\color{black}
    ordered} ${x_1, x_2 \in \mathcal{X}}$}, \mok{i.e., $x_1 \leq x_2$
  or $x_2 \leq x_1$}, is an inclusion function for $\mok{g}$. We refer
to this inclusion as the \emph{mixed-monotone inclusion function} of
\kj{$g$}.
Furthermore, $\mok{B} \in
\mathbb{D}_{n}$ 
is a binary diagonal matrix \mok{that identifies the vertex of the
  interval $[x_1,x_2]$ (or $[x_2,x_1]$) that minimizes (or maximizes) the
  JSS function $h$ in the case that $x_1 \leq x_2$ (or $x_2 \leq
  x_1$)}, 
  and {can be} computed as follows:
$\mok{B}=\textstyle{\mathrm{diag}}(\max(\textstyle{\mathrm{sgn}}({\overline{J}^{\mok{g}}}),\mathbf{0}_{1,{n}}))$.
Finally, $h_{d}$ {is} {tight},
i.e., $$\overline{h}_{\mathcal{X}}\hspace{-.1cm}\triangleq\hspace{-.1cm}
h_{d}(\underline{x},\overline{x})\hspace{-.1cm}=\hspace{-.1cm}\min_{x
  \in \mathcal{X}} h(x), \
\underline{h}_{\mathcal{X}}\hspace{-.1cm}\triangleq\hspace{-.1cm}
h_{d}(\overline{x},\underline{x})\hspace{-.1cm}=\hspace{-.1cm}\max_{x
  \in \mathcal{X}} h(x).$$
\end{prop}

From now on, $[g](\mathcal{X})$ denotes the mixed-monotone inclusion
function of $g$ on $\mathcal{X}$, unless otherwise specified.
\vspace{-.1cm}
\subsection{Guaranteed Privacy-Preserving \mok{Mechanism}}\label{sec:private-mech}
We are ready to introduce a guaranteed privacy-preserving map (or mechanism) through the following theorem.
\begin{thm}[Guaranteed Privacy of Functional
  Perturbation]\label{thm:guar_priv}
  Let $F=\{f_i(x)\}_{i=1}^N=\{h_i(x)+m_ix\}_{i=1}^N$ be the \mok{JSS decompositions of the set of functions} $f_i:\mathcal{X}_0
  \triangleq [\underline{x}_0,\overline{x}_0] \in \mathbb{IR}^n\hspace{-.1cm} \to \hspace{-.1cm} 
  \mathbb{R},i\in \until{n}$, 
    based on
  Proposition~\ref{prop:JSS_decomp}. Suppose \mok{Assumption \ref{ass:mixed_monotonicity} holds,} $\mathcal{X}_0$ is not a
  singleton, $\tilde{m}_i^\top \in \mathbb{R}^{n}$, and 
\begin{align}\label{eq:epsilon}
  \epsilon_i \triangleq \beta(f_i,\tilde{m}_i,\mathcal{X}_0,\delta_i)
  \hspace{-.1cm}\triangleq\hspace{-.1cm}\min_{m \in
    \mathbf{M}_{f_i}}\frac{\ln{(\frac{\Delta^{h_i}_{\mathcal{X}_0}
        +|\hat{m}_i|\Delta+2\delta_i}{\Delta^{h_i}_{\mathcal{X}_0}+|\hat{m}_i|\Delta
        })}}{\delta_i},
 \end{align}
 where $\Delta \triangleq
 \overline{x}_0-\underline{x}_0, \hat{m}_i \triangleq m_i+\tilde{m}_i$,
 $\Delta^ {h_i}_{{\mathcal{X}_0}}\triangleq
 \overline{h}_{i}-\underline{h}_{i}$ and
   $\mathbf{M}_{f_i},\overline{h}_{i},\underline{h}_{i}$
 are given in Propositions \ref{prop:JSS_decomp} and
 \ref{prop:mm_inclusion}.
 Then, 
 the mapping $\mathcal{M}:L_2(D)^N \times \mathcal{X}_0 \to
 \mathbb{IR}^N$ defined as
\begin{align}\label{eq:decomposition}
\begin{array}{c}
  \mathcal{M}(F,\mathcal{X}_0)\hspace{-.1cm}=\hspace{-.1cm}[\mathcal{G}](\mathcal{X}_0), [\mathcal{G}]\triangleq [[g_1]^\top,\dots,[g_N]^\top]^\top,\\
   g_i(x)\triangleq f_i(x)+\tilde{m}_ix, \; \forall x \in \mathcal{X}_0 , \forall i \in \until{n}
   \end{array}
\end{align}
satisfies $\epsilon$-guaranteed privacy where $\epsilon=\max_{i \in
  \until{n}}\epsilon_i$, with respect to the vicinity:
  
  \vspace{-.3cm}
  {\small
 \begin{align}\label{eq:vicinity}
   \hspace{-.25cm} \mathcal{V} \hspace{-.1cm} \triangleq \hspace{-.15cm}
   \{f'_i \hspace{-.1cm} \in \hspace{-.1cm} L_2(D) |
   \forall i \hspace{-.1cm}\in\hspace{-.1cm} \{1,\hspace{-.05cm}\dots,\hspace{-.05cm}N\},
   |{f}'_i(x)\hspace{-.1cm}-\hspace{-.1cm}{f}_i(x)| \hspace{-.1cm}\leq
   \hspace{-.1cm}\delta_{{i}}, \forall x \hspace{-.1cm}\in
   \hspace{-.1cm}\mathcal{X}_0 \}.
 \end{align}}
  \mok{We call $\{\tilde{m}_i\}_{i=1}^N$ the ``perturbation slopes''
    of the mechanism throughout the paper.}
 \end{thm}
 From Theorem \ref{thm:guar_priv}, the privacy \mok{gap}
$\epsilon_i=\beta(f_i,\tilde{m}_i,\mathcal{X}_0,\delta_i)$ is a decreasing
function of $\delta_i$, \mok{i.e.,} {\color{black} the smaller $\epsilon$ is,
  the harder it will be to distinguish the solution to problems with
  functions in a vicinity $\mathcal{V}$. }
  {\color{black} Note that, this result shows that by
  making $\delta_i$ large, i.e., by allowing the distance between the
  perturbed and the true function become larger, we can make
  $\epsilon_i$ small and thus increase privacy. However, this will
  directly impact the distance between the optimizers of the
  corresponding problems, which results in a loss of the quality of the
  solution---in the sense of being close to the original one.}
  This intuitively characterizes a trade-off between
privacy and accuracy, which will be discussed
later. 
\begin{proof} 
  \mok{With a slight abuse of notation, let
    $\mathcal{M}(f_i,\mathcal{X}_0)$ denote the $i^\textup{th}$
    argument of the interval vector $ \mathcal{M}(F,\mathcal{X}_0)$.}
  It follows from \eqref{eq:decomposition} and Propositions
  \ref{prop:JSS_decomp} and \ref{prop:mm_inclusion}
    that
\begin{align*}
  \mathcal{M}(f_i,\mathcal{X}_0)=[\underline{h}_{i}\hspace{-.1cm}+\hspace{-.1cm}m_i^\oplus
  \underline{x}_0\hspace{-.1cm}-\hspace{-.1cm}m_i^\ominus
  \overline{x}_0,\overline{h}_{i}\hspace{-.1cm}+\hspace{-.1cm}m_i^\oplus
  \overline{x}_0\hspace{-.1cm}-\hspace{-.1cm}m_i^\ominus
  \underline{x}_0]. 
\end{align*}
 Consequently, 
\begin{align}\label{eq:appr_range}
  \tdiam{(\mathcal{M}(f_i,\mathcal{X}_0))} =
  \Delta^{h_i}_{\mathcal{X}_0}+|\hat{m}_i|\Delta.
\end{align}
  On the other hand,
\mok{\eqref{eq:vicinity}} implies that for any $f'_i$ in the vicinity of
$f_i$, $-\delta_i+f_i(x) \leq f'_i(x) \leq \delta_i+f_i(x), \forall x \in
\mathcal{X}_0$. 
  By adding the
perturbation functions $\tilde{m}_ix$ \mok{and using the JSS
  decomposition of $f_i$ for an arbitrary $m_i \in \mathbf{M}_{f_i}$,}
we obtain $-\delta_i+h_i(x)+{\color{black} \kj{m_i}x} +\tilde{m}_ix \leq
f'_i(x)+\tilde{m}_ix \leq \delta_i+h_i(x)+m_ix+\tilde{m}_ix$, \mok{which} implies:

\vspace{-.3cm}
{\small
\begin{align*}
  \mathcal{M}(f'_i,\mathcal{X}_0) \hspace{-.1cm} \subseteq
  [-\delta_i\hspace{-.1cm}+\hspace{-.1cm}\underline{h}_{i}\hspace{-.1cm}+\hspace{-.1cm}\hat{m}_i^\oplus
  \underline{x}_0\hspace{-.1cm}-\hspace{-.1cm}\hat{m}_i^\ominus
  \overline{x}_0
  ,\delta_i\hspace{-.1cm}+\hspace{-.1cm}\overline{h}_{i}\hspace{-.1cm}+\hspace{-.1cm}\hat{m}_i^\oplus
  \overline{x}_0\hspace{-.1cm}-\hspace{-.1cm}\hat{m}_i^\ominus
  \underline{x}_0].
\end{align*}
}This, together with the fact $\tdiam{(\mathcal{M}(f'_i,\mathcal{X}_0)
  \cap \kj{\mathcal{I}})} \leq \tdiam{(\mathcal{M}(f'_i,\mathcal{X}_0))}$ \mok{for any interval \kj{$\mathcal{I}$}},
results \mok{in} 
\begin{align}\label{eq:epsilon_2}
  \tdiam{(\mathcal{M}(f'_i,\mathcal{X}_0) \cap \kj{\mathcal{I}})} \leq
  2\delta_{{i}}+\Delta^{h_i}_{\mathcal{X}_0}+|\hat{m}_i|\Delta.
\end{align}
Finally, it follows from \eqref{eq:epsilon} , \eqref{eq:appr_range} and \eqref{eq:epsilon_2} that
\begin{align*}
  \frac{\tdiam{(\mathcal{M}(f'_i,\kj{\mathcal{X}_0}) \cap
      \kj{\mathcal{I}})}}{\tdiam{(\mathcal{M}(f_i,\kj{\mathcal{X}_0}) 
     )}} \hspace{-.1cm}\leq\hspace{-.1cm}\min_{m \in
    \mathbf{M}_{f_i}}
  \frac{2\delta_{i}+\Delta^{h_i}_{\mathcal{X}_0}+|\hat{m}_i|\Delta}{\Delta^{h_i}_{\mathcal{X}_0}+|\hat{m}_i|\Delta}\hspace{-.1cm}=e^{\alpha_i},
\end{align*}
with $\alpha_i \triangleq \epsilon_i\delta_i$, \mok{which implies
  $\tdiam{(\mathcal{M}(f'_i,\mathcal{X}_0) \cap \kj{\mathcal{I}})} \leq
  e^{\epsilon_i \delta_i} \tdiam{(\mathcal{M}(f_i,\mathcal{X}_0)
    )}$. Taking the maximum of both sides on $i$ returns the results
  in
  \eqref{eq:guaranteed_privacy_inequality}}. 
\end{proof}
\subsection{Tractable Computation of an Optimal Perturbation
 } \label{sec:optimal-perturbation}
According to \eqref{eq:epsilon}, an important factor that affects the
privacy gap, in addition to $\delta_i$, is the choice of the
\emph{perturbation slope}, $\tilde{m}_i$. In this subsection,
we {\color{black} investigate what the best choice} for $\tilde{m}_i$
is. {\color{black} In the next section, we further discuss how to
  leverage the choice of $\delta_i$ so that privacy is still ensured
  for functions in a particular vicinity. }




{\color{black} It is reasonable to choose
  the  perturbation slopes in such a way that the difference between the 
  minimum of the perturbed function and the true function is
  reduced as much as possible. An approximated upper bound of this
  difference can be obtained by leveraging mixed-monotone inclusion
  functions (cf.~\mok{Proposition~\ref{prop:mm_inclusion}}). Given the
    common nonlinear terms $h$, this bound can be minimized by
    reducing the difference of the linear terms for all possible
    subintervals of the initial interval domain $\mathcal{X}_0$.}
 This results into the
following 
robust optimization problem:
\begin{align}\label{eq:robust_measure}
\begin{array}{rl}
  \tilde{m}^*=&\argmin\limits_{\tilde{m} \in \mathbb{R}^{1 \times n}} |(\hat{m}^{\oplus}-m^\oplus)\underline{x}-(\hat{m}^{\ominus}-m^\ominus)\overline{x}| \\
  & \forall [\underline{x},\overline{x}] \subseteq [\underline{x}_0,\overline{x}_0],
\end{array}
\end{align} 
where $\hat{m} \triangleq m +\tilde{m}$. Note that the choice of
$\tilde{m}$ through \eqref{eq:robust_measure} is not necessarily
optimal in the sense that it minimizes the privacy gap or
\mok{ maximizes the accuracy of the chosen optimization
  method to solve the problem}. 
However, given that our goal is to choose the perturbation slope
\emph{independently} of the chosen optimization method, 
  it is reasonable to use
\eqref{eq:robust_measure}.
  \moh{Moreover}, a significant advantage of designing the perturbation
slope \mok{through}~\eqref{eq:robust_measure} is that it
provides us with a \emph{tractable approach} to obtain 
$\tilde{m}$, \mok{which} is done via the transformation of~\eqref{eq:robust_measure}
into a linear program (LP), discussed in the following lemma.
\begin{lem}[Tractable Computation of Perturbation
  \mok{Slopes}]\label{lem:tractable} 
    The robust
  optimization problem in \eqref{eq:robust_measure} can be
  equivalently reformulated to the following linear program:
\begin{align}\label{eq:LP}
\begin{array}{rll}
  &\min\limits_{\xi \in \mathbb{R}^{2n+1},p_1,p_2 \in \mathbb{R}^{3n}} c^\top \xi \\
  &\st \quad \mok{\Lambda}\xi \leq l, \ p_1^\top d \leq 0, \ p_2^\top d \leq 0 ,\\
  &\quad  \mok{\Gamma}^\top p_1=\xi, \ -\mok{\Gamma}^\top p_2=\xi,p_1 \geq \mathbf{0}_{3n}, \ p_2 \geq \mathbf{0}_{3n},   
\end{array}
\end{align}
where $d_1\hspace{-.1cm}=\hspace{-.1cm}d_2
  \hspace{-.1cm}\triangleq\hspace{-.1cm} \begin{bmatrix}
    \overline{x}_0^\top & \underline{x}_0^\top &
    \mathbf{0}^\top_{n} \end{bmatrix}^\top\hspace{-.1cm},l
  \hspace{-.1cm}\triangleq\hspace{-.1cm} \begin{bmatrix}m^\oplus &
    m^\ominus & 0  \end{bmatrix}^\top\hspace{-.1cm},$
\begin{align*}
  \Gamma\hspace{-.1cm}\triangleq\hspace{-.15cm}\begin{bmatrix} \mathbf{0}_{n \times
      n} & -I_{n} & \mathbf{0}_{n} \\ -I_{n } & \mathbf{0}_{n \times
      n} & \mathbf{0}_{n} \\ I_{n } & I_{n} &
    \mathbf{0}_{n} \end{bmatrix}\hspace{-.1cm},\mok{\Lambda}\hspace{-.1cm}\triangleq\hspace{-.15cm}\begin{bmatrix} -I_{n} &
    \mathbf{0}_{n \times n} & \mathbf{0}_{n} \\ \mathbf{0}_{n \times
      n} & -I_{n} & \mathbf{0}_{n} \\ \mathbf{0}^\top_{n} &
    \mathbf{0}^\top_{n} & -1 \end{bmatrix}\hspace{-.1cm},c
  \hspace{-.1cm}\triangleq\hspace{-.1cm} \begin{bmatrix}
    \mathbf{0}^\top_{2n} & 1 \end{bmatrix}^\top\hspace{-.2cm}.
\end{align*}
Moreover, 
\begin{align}\label{eq:optimal_m}
  \tilde{m}^*=(\xi^*)^\top \begin{bmatrix} I_{n} & -I_{n} & \mathbf{0}^\top_{n} \end{bmatrix}^\top,
\end{align} 
where $\tilde{m}^*$ and $\xi^*$ are solutions to the robust
optimization in \eqref{eq:robust_measure} and the LP in \eqref{eq:LP},
respectively.
\end{lem}
\begin{proof}
  First, note that the robust program in \eqref{eq:robust_measure} can
  be equivalently written as follows, with $\hat{m} \triangleq
  m+\tilde{m}$:
\begin{align*}
\begin{array}{rl}
  &\min_{\{\tilde{m} \in \mathbb{R}^{1 \times n}, \theta \geq 0\}} \theta \\
  &\text{s.t.} -\theta \hspace{-.1cm}\leq\hspace{-.1cm}
  (\hat{m}^{\oplus}\hspace{-.1cm}-\hspace{-.1cm}m^\oplus)\underline{x}\hspace{-.1cm}-\hspace{-.1cm}(\hat{m}^{\ominus}\hspace{-.1cm}-\hspace{-.1cm}m^\ominus)\overline{x}
  \hspace{-.1cm}\leq\hspace{-.1cm} \theta, \forall
  [\underline{x},\overline{x}] \subseteq
  [\underline{x}_0,\overline{x}_0].
\end{array}
\end{align*}  
In turn, by considering the change of variables $\eta \triangleq
\hat{m}^\oplus-m^\oplus, \rho \triangleq \hat{m}^\ominus-m^\ominus,
\xi \triangleq [\eta \ \rho \ \theta]^\top, a_1 \triangleq
[\underline{x}^\top -\overline{x}^\top -1]^\top,a_2 \triangleq
[-\underline{x}^\top \overline{x}^\top -1]^\top$, the latter can be
reformulated as:
\begin{align*}
\begin{array}{rl}
  &\min_{\{\xi\}} c^\top \xi \\
  &\st  \mok{\Lambda} \xi \hspace{-.1cm}\leq\hspace{-.1cm} l, [a_1 \
  a_2]^\top \xi \leq \mathbf{0}_2, \forall a_1, a_2 \ \st 
  \mok{\Gamma}a_1\hspace{-.1cm} \leq \hspace{-.1cm} d, -\mok{\Gamma}a_2 \hspace{-.1cm} \leq
  \hspace{-.1cm} d,  
\end{array}
\end{align*}
with $c,\mok{\Gamma},\mok{\Lambda},d$ and $l$ given under \eqref{eq:LP}. Furthermore, by
\cite[Section 1.2.1]{ben2009robust}, the above robust LP can be
equivalently cast as the regular LP in \eqref{eq:LP}. Finally, with
$\xi^* = [\eta^* \ \rho^* \ \theta^*]^\top$ being a solution to
\eqref{eq:LP}, $(\xi^*)^\top \begin{bmatrix} I_{n} & -I_{n} &
  \mathbf{0}^\top_{n} \end{bmatrix}^\top=\eta^* -
\rho^*=\hat{m}^{*\oplus}-m^\oplus- (
\hat{m}^{*\ominus}-m^\ominus)=\hat{m}^*-m=\tilde{m}^*$.
\end{proof}

\subsection{Guaranteed Private Mechanism and Accuracy
  Analysis}\vspace{-0.05cm}
In this subsection,  {\color{black}and in view of the results of the
  previous section, we slightly modify our privacy
  mechanism, and investigate its accuracy when applied to a
  distributed optimization setting.} 
  In particular, we show that, regardless of the distributed
and convergent optimization algorithm employed, a perturbation of the
problem objective functions via the map of the form 
of Theorem~\ref{thm:guar_priv} ensures guaranteed privacy, \kj{while remains
  reasonably accurate. In other words, we show that there is a
  computable and reasonably tight upper bound for the error caused by
  perturbations, regardless of the chosen optimization algorithm}.
\kj{To do so}, we require that each agent $i \in \until{N}$ computes
the function $g_i$, where $\forall x \in \mathcal{X}_0$, \mok{$g(x)
  \triangleq \sum_{i=1}^Ng_i(x)$,} 
{\color{black} \begin{align}\label{eq:local_perturbation} 
  g_i(x)&=f_i(x)+ \tilde{m}_i x 
\end{align}
where, as we explain next, $\tilde{m}_i$ is \kj{constrained by a mild condition that is characterized by} 
$\tilde{m}_i^* =
(\xi^*_i)^\top \begin{bmatrix} I_{n} & -I_{n} &
  \mathbf{0}^\top_{n} \end{bmatrix}^\top,$
and $\xi^*_i$ solves \eqref{eq:LP} after replacing $m$ with
$m_i$. After this process, agents implement \emph{any} distributed
optimization algorithm with the modified objective functions
$\{g_i\}_{i=1}^N$. Let 
 }
  \vspace{-.4cm} {\small
\begin{align}\label{eq:g_opt}
\begin{array}{rl}
  \hspace{-.4cm}{\mathbb{X}}_g\hspace{-.1cm}=\hspace{-.1cm}\displaystyle{\argmin_{x \in \mathcal{X}_0}}\hspace{-.1cm} \sum_{i=1}^N g_i(x)  \triangleq \{\tilde{x}^* \hspace{-.1cm}\in \hspace{-.1cm}
  \mathcal{X}_0 \,|\, g(\tilde{x}^*)\leq g(x), \forall x
  \hspace{-.1cm}\in \hspace{-.1cm}\mathcal{X}_0\},\\
  \hspace{-.4cm}{\mathbb{X}}_f\hspace{-.1cm}=\hspace{-.1cm}\displaystyle{\argmin_{x \in \mathcal{X}_0}}\hspace{-.1cm} \sum_{i=1}^N f_i(x)
 \triangleq\{{x}^* \hspace{-.1cm}\in \hspace{-.1cm} \mathcal{X}_0\, | \,f({x}^*)\leq f(x), \forall x \hspace{-.1cm}\in \hspace{-.1cm}\mathcal{X}_0\},
\end{array}
\end{align}}

\vspace{-.4cm}
\noindent denote the set of possible outputs of the distributed
algorithm, and the \emph{set} of optimizers of the original problem
\eqref{eq:opt_init}, respectively.  The following theorem
characterizes the accuracy and privacy of the {\color{black}
corresponding perturbation \kj{introduced in}
  Section~\ref{sec:optimal-perturbation}}, 
\kj{in terms of a computably-tractable upper bound for the errors
  incurred when} {\color{black} using
  them.} 

\begin{thm}[{\color{black} Guaranteed-Private Mechanism and Its
    Accuracy}] \label{thm:accuracy}
  Consider a group of $N$ agents that aim to collectively solve the
  distributed nonconvex optimization \eqref{eq:opt}. Suppose
  Assumption~\ref{ass:mixed_monotonicity} holds, 
  denote $\Delta\hspace{-.1cm} \triangleq \hspace{-.1cm}
    \overline{x}_0 \hspace{-.1cm}-\hspace{-.1cm}\underline{x}_0$,
  and define 
  \begin{align}\label{eq:delta}
    \delta^*_{i} \triangleq \max (|\tilde{m}^{*\oplus}_i\overline{x}_0-
    \tilde{m}^{*\ominus}_i\underline{x}_0|,|\tilde{m}^{*\oplus}_i\underline{x}_0-\tilde{m}^{*\ominus}_i\overline{x}_0|),
  \end{align}
  with $\tilde{m}^*_i$ given in \eqref{eq:local_perturbation}. Then,
\vspace{-.2cm}
 \begin{enumerate}[(i)]
 \item {\color{black} For any $\delta_i \ge \delta^*_{i}$, the
     family $G=\{g_i\}_{i=1}^N$ with an arbitrary perturbation
     $\tilde{m}_i$ such that $\tilde{m}_i \Delta \le \delta_i^*$, belongs
     to a $\delta \ge \max_i\delta_i^*$ vicinity of the family $F =
     \{f_i\}_{i=1}^N$. \kj{Moreover}, the mapping $\mathcal{M}$ given \kj{in}
     Theorem~\ref{thm:guar_priv} for this class of perturbations, is
     $\epsilon = \max_{i \in \{1,\dots, N\}}\epsilon_i$-guaranteed
     private, where $\epsilon_i =
     \beta(f_i,\tilde{m}_i,\mathcal{X}_0,\delta_{i})$.
    \label{item:privacy}}
\item The (worst-case) accuracy error, defined
  as:\label{item:accuracy}
\begin{align}\label{eq:error}
  e(\{f_i\}_{i=1}^N,\{\tilde{m}_i\}_{i=1}^N,\mathcal{X}_0)\hspace{-.1cm}\triangleq\hspace{-.2cm}\max_{x^*
    \in \mathbb{X}_f,\tilde{x}^* \in \mathbb{X}_g}\hspace{-.2cm}
  \|x^*\hspace{-.1cm}-\hspace{-.1cm}\tilde{x}^*\|_{\infty}
\end{align} 
satisfies the following upper
bound 
\begin{equation*}
  e(\{f_i\}_{i=1}^N,\{\tilde{m}_i\}_{i=1}^N,\mathcal{X}_0)
\leq \textstyle{\mathrm{UB}}, \ \text{ where} \nonumber\\
\end{equation*}
 \begin{align}\label{eq:error_up}
 &\begin{array}{rl}
   &\tUB =\max\limits_{\{y \in \mathcal{X}_0,z\in \mathcal{X}_0,\theta \in \mathbb{R}_{\geq 0}\}}\theta\\
   & \st  -\theta \mathbf{1}_n \hspace{-.1cm}\leq\hspace{-.1cm} y\hspace{-.1cm}-\hspace{-.1cm}z \hspace{-.1cm}\leq\hspace{-.1cm} \theta \mathbf{1}_n,
   \tilde{m}_i(y\hspace{-.1cm}-\hspace{-.1cm}z) \hspace{-.1cm}\leq\hspace{-.1cm} {0}, 1 \hspace{-.1cm} \leq \hspace{-.1cm}  i \hspace{-.1cm} \leq \hspace{-.1cm} N,
 \end{array}
 \end{align}  
 and $\mathbb{X}_f,\mathbb{X}_g$ are given in \eqref{eq:g_opt}.
\end{enumerate}
\end{thm}
\begin{proof}
  To prove \eqref{item:privacy}, note that by
  \eqref{eq:local_perturbation} and \cite[Lemma
    1]{efimov2013interval}, $|g_i(x)-f_i(x)|=\tilde{m}_ix \in
  [\tilde{m}^{\oplus}_i\underline{x}_0-\tilde{m}^{\ominus}_i\overline{x}_0,\tilde{m}^{\oplus}_i\overline{x}_0-\tilde{m}^{\ominus}_i\underline{x}_0]$,
  $\forall x \in \mathcal{X}_0$, implying that {\color{black} $|g_i(x)-f_i(x)| \leq
  |\tilde{m}_i| \Delta  \leq \delta^*_{i} 
  \triangleq \max
  (|\tilde{m}^{*\oplus}_i\overline{x}_0-\tilde{m}^{*\ominus}_i\underline{x}_0|,|\tilde{m}^{*\oplus}_i\underline{x}_0-\tilde{m}^{*\ominus}_i\overline{x}_0|)$
  $
  \leq \delta_i$,
  $\forall x \in \mathcal{X}_0$.} Then, \eqref{item:privacy} follows
  from applying Theorem \ref{thm:guar_priv} on each $f_i$. 

  To prove \eqref{item:accuracy}, first note that for any
  $(x^*,\tilde{x}^*) \in \mathbb{X}_f \times \mathbb{X}_g$, the
  following holds:
\begin{align*}
f(x^*)&=h(x^*)+mx^*\leq f(\tilde{x}^*)=f(\tilde{x}^*)+\tilde{m}\tilde{x}^*-\tilde{m}\tilde{x}^*\\
         &=\hspace{-.1cm}g(\tilde{x}^*)\hspace{-.1cm}-\hspace{-.1cm}\tilde{m}\tilde{x}^*
         \hspace{-.15cm}\leq\hspace{-.1cm}
         g(x^*)\hspace{-.1cm}-\hspace{-.1cm}\tilde{m}\tilde{x}^*\hspace{-.1cm}=\hspace{-.1cm}h(x^*)\hspace{-.1cm}+\hspace{-.1cm}mx^*\hspace{-.1cm}+\hspace{-.1cm}\tilde{m}x^*\hspace{-.15cm}-\hspace{-.1cm}\tilde{m}\tilde{x}^*,
\end{align*}
where the first and second inequalities follow from the fact that
$x^*$ and $\tilde{x}^*$ are minimizers of $f$ and $g$,
respectively. Conclusively, given any $(x^*,\tilde{x}^*) \in
\mathbb{X}_f \times \mathbb{X}_g$ and any perturbation slope, 
it is necessary that
$\tilde{m}(x^*-\tilde{x}^*) \leq 0$. By this and defining
$\theta=\|x^*-\tilde{x}^*\|_{\infty}$, the program in \eqref{eq:error}
is equivalent to
\begin{align}\label{eq:error_up_2}
 \begin{array}{rl}
 &e(\{f_i\}_{i=1}^N,\{\tilde{m}_i\}_{i=1}^N,\mathcal{X}_0)\hspace{-.1cm}=\hspace{-.1cm}\max\limits_{\{y \in \mathbb{X}_f,z\in \mathbb{X}_g,\theta \in \mathbb{R}_{\geq 0}\}}\theta\\
 & \text{s.t.}  -\theta \mathbf{1}_n \hspace{-.1cm}\leq\hspace{-.1cm} y\hspace{-.1cm}-\hspace{-.1cm}z \hspace{-.1cm}\leq\hspace{-.1cm} \theta \mathbf{1}_n,
 \tilde{m}_i(y\hspace{-.1cm}-\hspace{-.1cm}z) \hspace{-.1cm}\leq\hspace{-.1cm} {0}, 1 \hspace{-.1cm} \leq \hspace{-.1cm}  i \hspace{-.1cm} \leq \hspace{-.1cm} N.
 \end{array}
 \end{align}  
 Finally, comparing \eqref{eq:error_up} and \eqref{eq:error_up_2}
 indicates that the optimal value of the former is an upper bound for
 the optimal value of the latter since the feasible set of the latter
 is a subset of the one for the former, i.e., $\mathbb{X}_f \times
 \mathbb{X}_g \subseteq \mathcal{X}_0 \times \mathcal{X}_0$.
\end{proof}
\kj{As a consequence of Theorem \ref{thm:accuracy}, since the choice
  of $\tilde{m}_i$ is almost arbitrary (constrained by the mild
  condition $\tilde{m}_i \Delta \le \delta_i^*$), then, it is very
  unlikely for an adversary to know/guess $\tilde{m}_i$.  From this
  perspective, the process remains private, with the
  level of privacy given in \eqref{eq:epsilon}, while the entire
  process remains accurate with and error less than the upper bound
  given in \eqref{eq:error_up}. Furthermore, $\tilde{m}_i \Delta \le
  \delta_i^*$ can be interpreted as a significantly weaker counterpart
  of the required conditions on the perturbation noise in differential
  privacy, e.g., in \cite{nozari2016differentially}.}
It is \kj{also} worth emphasizing that the computed accuracy error upper bound,
though might be conservative depending on 
objective function {\color{black} and constraints},
but on the other hand, it provides an upper bound for the accuracy
error \emph{regardless of the chosen algorithm}. This can be
interpreted as an additional degree of resiliency or the proposed
privacy-preserving mechanism against perturbing the
selected optimization algorithms. 

  \vspace{-0.1cm}
\section{Illustrative Example}
To illustrate the effectiveness of our approach, we considered a
nonconvex distributed optimization example from
\cite{tatarenko2017non}, which is in the from of \eqref{eq:opt}, with
$n=1,N=3$ and $\mathcal{X}_0=[-10,10]$,  where
$f_1(x)=(x^3-16x)(x+2)$, $f_2(x)=(0.5x^3+x^2)(x-4)$ and
$f_3(x)=(x+2)^2(x-4)$, with the global optimizer $x^*=[2.62 \ 2.62 \
2.62]^\top$. We implemented the following 5 algorithms:
the \emph{nonconvex distributed optimization} (NDO) proposed in
\cite{tatarenko2017non}, the \emph{nonconvex decentralized gradient
  descent} (NDGD) approach in \cite{zeng2018nonconvex}, the
\emph{distributed nonconvex constrained optimization} (DNCO) method
introduced in \cite{scutari2019distributed}, the \emph{distributed
  nonconvex first-order optimization} (DNCFO) algorithm in
\cite{sun2019distributed}, and the \emph{distributed zero-order
  algorithm} (DZOA) from \cite{tang2020distributed}.

Using Lemma \ref{lem:tractable}, \mok{a set of} tractable perturbation
slopes was obtained as $\tilde{m}^*=\{\tilde{m}^*_1, \tilde{m}^*_2,
\tilde{m}^*_3\}=\{0.52,0.73, 0.38\}$ through solving the LP in
\eqref{eq:LP}. {\color{black} The $\epsilon$-guaranteed privacy gaps
  of the mechanism defined in Theorem~\ref{thm:accuracy} when using
  $\tilde{m}_i^*$ are given by} 
$\{\epsilon^*_1,\epsilon^*_2, \epsilon^*_3\}=\{0.14,0.32,0.68\}$,
respectively. Further, to study the compromise between privacy and
accuracy, we randomly picked $50$ samples of $\tilde{m}$ chosen from
the normal distribution $\mathcal{N}(\tilde{m}^*,1)$ and applied
Theorem \ref{thm:accuracy}, aiming to compute the corresponding
$\epsilon=\max_{\{i=1,\dots,3\}}\epsilon_i$ and the worst-case
accuracy error, i.e.,
$e(\{f_i\}_{i=1}^3,\{\tilde{m}_i\}_{i=1}^3,\mathcal{X}_0)$
(cf.~\eqref{eq:error}), as well as the theoretical error upper bound
\mok{$\tUB$} (cf.~\eqref{eq:error_up}) for each sampled
$\tilde{m}$. For illustration, in Figure~\ref{fig:privacy_accuracy}
the computed values for the privacy gap ($\epsilon$) are sorted in an
ascending order along the horizontal axis, \mok{which}, as can be
observed, resulted in descending (decreasing) corresponding errors, \mok{after
  algorithms converge. As can be observed, at the highest privacy
  ($\epsilon=0.116$), we obtain the lowest accuracy (i.e., highest
  accuracy error) which the can be very tightly approximated with the
  theoretical upper bound for all the optimization
  algorithms. Moreover, eventually, when privacy is the lowest
  ($\epsilon=0.155$), all the algorithms converge to the lowest accuracy
  error ($e=0.01$), which \kj{again} can be reasonably over-approximated by the
  corresponding theoretical upper bound ($\tUB=0.025$).} 
Finally, as Figure \ref{fig:privacy_accuracy} shows, the theoretical
upper bound is independent of the chosen optimization method \kj{and is} bounding
all error sequences, {\color{black} noting that the different
  convergence results depend on the type of nonconvex method
  employed.}
\begin{figure}[t!]
  \centering
  {\includegraphics[width=0.98\columnwidth]{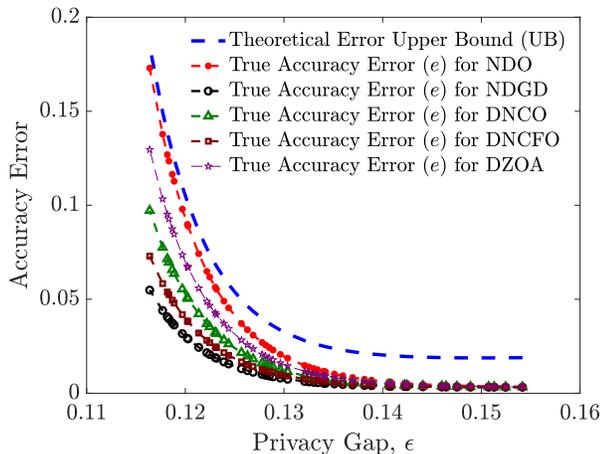}}
  \caption{{\small Theoretical accuracy error upper bound ($\tUB$) computed according to \eqref{eq:error_up} in Theorem \ref{thm:accuracy} with the perturbation slope $\tilde{m}^*$ obtained by solving the LP in \ref{eq:LP}, as well as true accuracy error $\|x^*-\tilde{x}^*\|_{\infty}$ for $50$ randomly sampled perturbations $\tilde{m}$ from the normal distribution $\mathcal{N}(\tilde{m}^*,1)$ obtained by applying the nonconvex distributed optimization algorithms NDO \cite{tatarenko2017non}, NDGD \cite{zeng2018nonconvex}, DNCO \cite{scutari2019distributed}, DNCFO \cite{sun2019distributed} and DZOA \cite{tang2020distributed}.}}
  \label{fig:privacy_accuracy}
\end{figure}
 \vspace{-.1cm}
 \section{Conclusion and \md{Future Work}} \label{sec:conclusion} This
 paper introduced a novel notion of guaranteed privacy for a
 {\color{black} broad class of} \kj{differentiable locally
   Lipschitz} nonconvex distributed optimization problems. We showed how
 this property holds for a deterministic type of perturbation
 mechanisms, which exploit the Jacobian sign-stability of the problem
 objective functions. \kj{Furthermore,} 
 using robust optimization techniques, a tractable approach was
 \kj{provided} to further restrict \mok{the mechanism} in a way that
 allows for the quantification of the accuracy bounds of the
 method. 
 In particular, these bounds were shown to be decreasing with respect
 to the privacy gap, \kj{as illustrated through simulations}. \mok{Future \mok{work} will consider utilizing
   the results in Theorem \ref{thm:guar_priv} to design
   privacy-preserving estimation, verification and resource/task allocation
   algorithms in networked CPS.} 
\bibliographystyle{unsrturl}

{\tiny
\bibliography{biblio}

\begin{thebibliography}{10}

\bibitem{cortes2016differential}
J.~Cort{\'e}s, G.E. Dullerud, S.~Han, J.~Le~Ny, S.~Mitra, and G.J. Pappas.
\newblock Differential privacy in control and network systems.
\newblock In {\em 2016 IEEE 55th Conference on Decision and Control (CDC)},
  pages 4252--4272. IEEE, 2016.

\bibitem{dwork2006calibrating}
C.~Dwork, F.~McSherry, K.~Nissim, and A.~Smith.
\newblock Calibrating noise to sensitivity in private data analysis.
\newblock In {\em Theory of cryptography conference}, pages 265--284. Springer,
  2006.

\bibitem{chaudhuri2011differentially}
K.~Chaudhuri, C.~Monteleoni, and A.D. Sarwate.
\newblock Differentially private empirical risk minimization.
\newblock {\em Journal of Machine Learning Research}, 12(3), 2011.

\bibitem{zhang2012functional}
J.~Zhang, Z.~Zhang, X.~Xiao, Y.~Yang, and M.~Winslett.
\newblock Functional mechanism: regression analysis under differential privacy.
\newblock {\em arXiv preprint arXiv:1208.0219}, 2012.

\bibitem{hall2013differential}
R.~Hall, A.~Rinaldo, and L.~Wasserman.
\newblock Differential privacy for functions and functional data.
\newblock {\em The Journal of Machine Learning Research}, 14(1):703--727, 2013.

\bibitem{wang2017differential}
Y.~Wang, Z.~Huang, S.~Mitra, and G.E. Dullerud.
\newblock Differential privacy in linear distributed control systems: Entropy
  minimizing mechanisms and performance tradeoffs.
\newblock {\em IEEE Transactions on Control of Network Systems}, 4(1):118--130,
  2017.

\bibitem{han2021numerical}
Y.~Han and S.~Mart{\'\i}nez.
\newblock A numerical verification framework for differential privacy in
  estimation.
\newblock {\em IEEE Control Systems Letters}, 6:1712--1717, 2021.

\bibitem{huang2015differentially}
Z.~Huang, S.~Mitra, and N.~Vaidya.
\newblock Differentially private distributed optimization.
\newblock In {\em Proceedings of the 2015 international conference on
  distributed computing and networking}, pages 1--10, 2015.

\bibitem{han2016differentially}
S.~Han, U.~Topcu, and G.J. Pappas.
\newblock Differentially private distributed constrained optimization.
\newblock {\em IEEE Transactions on Automatic Control}, 62(1):50--64, 2016.

\bibitem{hale2015differentially}
M.T. Hale and M.~Egerstedty.
\newblock Differentially private cloud-based multi-agent optimization with
  constraints.
\newblock In {\em 2015 American Control Conference (ACC)}, pages 1235--1240.
  IEEE, 2015.

\bibitem{ding2021differentially}
T.~Ding, S.~Zhu, J.~He, C.~Chen, and X.~Guan.
\newblock Differentially private distributed optimization via state and
  direction perturbation in multiagent systems.
\newblock {\em IEEE Transactions on Automatic Control}, 67(2):722--737, 2021.

\bibitem{li2020privacy}
Q.~Li, R.~Heusdens, and M.G. Christensen.
\newblock Privacy-preserving distributed optimization via subspace
  perturbation: A general framework.
\newblock {\em IEEE Transactions on Signal Processing}, 68:5983--5996, 2020.

\bibitem{ye2021differentially}
M.~Ye, G.~Hu, L.~Xie, and S.~Xu.
\newblock Differentially private distributed {N}ash equilibrium seeking for
  aggregative games.
\newblock {\em IEEE Transactions on Automatic Control}, 67(5):2451--2458, 2021.

\bibitem{nozari2016differentially}
E.~Nozari, P.~Tallapragada, and J.~Cort{\'e}s.
\newblock Differentially private distributed convex optimization via functional
  perturbation.
\newblock {\em IEEE Transactions on Control of Network Systems}, 5(1):395--408,
  2016.

\bibitem{wang2022verification}
Y.~Wang, H.~Sibai, M.~Yen, S.~Mitra, and Dullerud g.E.
\newblock Differentially private algorithms for statistical verification of
  cyber-physical systems.
\newblock {\em http://128.84.4.18/pdf/2004.00275}, 2022.

\bibitem{bertsekas2003convex}
D.~Bertsekas, A.~Nedic, and A.~Ozdaglar.
\newblock {\em Convex analysis and optimization}, volume~1.
\newblock Athena Scientific, 2003.

\bibitem{jaulinapplied}
L.~Jaulin, M.~Kieffer, O.~Didrit, and E.~Walter.
\newblock Applied interval analysis.
\newblock {\em ed: Springer, London}, 2001.

\bibitem{khajenejad2021tight}
M.~Khajenejad and S.Z. Yong.
\newblock Tight remainder-form decomposition functions with applications to
  constrained reachability and interval observer design.
\newblock {\em IEEE Transactions on Automatic Control, accepted, arXiv preprint
  arXiv:2103.08638}, 2022.

\bibitem{yang2019sufficient}
L.~Yang, O.~Mickelin, and N.~Ozay.
\newblock On sufficient conditions for mixed monotonicity.
\newblock {\em IEEE Transactions on Automatic Control}, 64(12):5080--5085,
  2019.

\bibitem{khajenejad2021guaranteed}
M.~Khajenejad, F.~Shoaib, and S.Z. Yong.
\newblock Guaranteed state estimation via indirect polytopic set computation
  for nonlinear discrete-time systems.
\newblock In {\em 2021 60th IEEE Conference on Decision and Control (CDC)},
  pages 6167--6174. IEEE, 2021.

\bibitem{ben2009robust}
A.~Ben-Tal, L.~El~Ghaoui, and A.~Nemirovski.
\newblock {\em Robust optimization}, volume~28.
\newblock Princeton university press, 2009.

\bibitem{efimov2013interval}
D.~Efimov, T.~Ra{\"\i}ssi, S.~Chebotarev, and A.~Zolghadri.
\newblock Interval state observer for nonlinear time varying systems.
\newblock {\em Automatica}, 49(1):200--205, 2013.

\bibitem{tatarenko2017non}
T.~Tatarenko and B.~Touri.
\newblock Non-convex distributed optimization.
\newblock {\em IEEE Transactions on Automatic Control}, 62(8):3744--3757, 2017.

\bibitem{zeng2018nonconvex}
J.~Zeng and W.~Yin.
\newblock On nonconvex decentralized gradient descent.
\newblock {\em IEEE Transactions on signal processing}, 66(11):2834--2848,
  2018.

\bibitem{scutari2019distributed}
G.~Scutari and Y.~Sun.
\newblock Distributed nonconvex constrained optimization over time-varying
  digraphs.
\newblock {\em Mathematical Programming}, 176(1):497--544, 2019.

\bibitem{sun2019distributed}
H.~Sun and M.~Hong.
\newblock Distributed non-convex first-order optimization and information
  processing: Lower complexity bounds and rate optimal algorithms.
\newblock {\em IEEE Transactions on Signal processing}, 67(22):5912--5928,
  2019.

\bibitem{tang2020distributed}
Y.~Tang, J.~Zhang, and N.~Li.
\newblock Distributed zero-order algorithms for nonconvex multiagent
  optimization.
\newblock {\em IEEE Transactions on Control of Network Systems}, 8(1):269--281,
  2020.

\end{thebibliography}
}

\end{document}